\documentclass[10pt,twoside,a4paper]{article}
\font\smallit=cmti10

\usepackage{amssymb,latexsym,amsmath,epsfig,amsthm,hyperref}

\makeatletter

\renewcommand\section{\@startsection {section}{1}{\z@}
{-30pt \@plus -1ex \@minus -.2ex}
{2.3ex \@plus.2ex}
{\normalfont\normalsize\bfseries\boldmath}}

\renewcommand\subsection{\@startsection{subsection}{2}{\z@}
{-3.25ex\@plus -1ex \@minus -.2ex}
{1.5ex \@plus .2ex}
{\normalfont\normalsize\bfseries\boldmath}}

\renewcommand{\@seccntformat}[1]{\csname the#1\endcsname. }

\makeatother

\newtheorem{theorem}{Theorem}

\newtheorem{corollary}{Corollary}

\theoremstyle{definition}

\newcommand{\ssum}[4][1.8em]%
    {\mathop{\sum_{#2}^{#3}}_{\makebox[#1][l]{$\scriptstyle{#4}$}}}

\newcommand{\heading}[1]{\smallskip\noindent%
                         \textbf{#1.}\hspace*{0.5em}\ignorespaces}
\newcommand{\adj}{^{\dagger}}\newcommand{\phadj}{^{\vphantom{\dagger}}}
\newcommand{\tr}[1]{\mathord{\mathrm{tr}}{\left( #1 \right)}}
\DeclareMathAlphabet{\columnfont}{OT1}{cmr}{bx}{it}
\DeclareMathAlphabet{\matrixfont}{OT1}{cmr}{bx}{sl}
\newcommand{\col}[1]{\mathord{\columnfont{#1}}}
\newcommand{\ca}{\col{a}}\newcommand{\cb}{\col{b}}\newcommand{\cc}{\col{c}}
\newcommand{\M}[1]{\mathord{\matrixfont{#1}}}
\newcommand{\unity}{\M{1}_{n}}
\newcommand{\MA}{\M{A}}\newcommand{\MB}{\M{B}}\newcommand{\MC}{\M{C}}

\begin{document}
\begin{center}
\uppercase{\bfseries Multiplicative functions arising from the study %
       of mutually unbiased bases}
\vskip 20pt
{\bfseries Heng Huat Chan}\\
{\smallit Department of Mathematics, %
  National University of Singapore, %
  Singapore, 119076, Singapore}\\
{\ttfamily matchh@nus.edu.sg}\\
\vskip 10pt
{\bfseries Berthold-Georg Englert}\\
{\smallit Centre for Quantum Technologies, %
  3 Science Drive 2, Singapore 117543, Singapore; %
  Department of Physics, National University of Singapore, %
  2 Science Drive 3, Singapore 117551, Singapore; %
  MajuLab, CNRS-UCE-SU-NUS-NTU, International Joint Research Unit, CNRS, %
  Singapore 117543, Singapore}\\
{\ttfamily englert@u.nus.edu}\\ 
\end{center}
 
\vskip 30pt

\centerline{\bfseries Abstract}
\noindent%
We embed the somewhat unusual multiplicative function, which was
serendipitously discovered in 2010 during a study of mutually unbiased bases
in the Hilbert space of quantum physics, into two families of
multiplicative functions that we construct as generalizations of that
particular example.
In addition, we report yet another multiplicative function, which is also
suggested by that example; it can be used to express the squarefree part of an
integer in terms of an exponential sum. 

\pagestyle{myheadings}
\markboth{Heng Huat Chan and Berthold-Georg Englert}%
{Multiplicative functions arising from mutually unbiased bases}

\thispagestyle{empty}
\baselineskip=12.875pt
\vskip 30pt

\section{Multiplicative functions}
We begin by recalling a few basic definitions and results. 
A real or complex valued function defined on the positive integers is called
an arithmetic function or a number-theoretic function. 

For any two positive integers $m$ and $n$,
we use $(m,n)$ to denote the greatest common divisor of $m$ and $n$.
An arithmetic function $f$ is called multiplicative if $f$ is not
identically zero and if 
\begin{equation*}
  f(mn)=f(m)f(n)\,\,\text{whenever $(m,n)=1$}.
\end{equation*}
It follows that $f(1)=1$ if $f$ is multiplicative.

In general, given an arithmetic function $f$, it is impossible to determine
$f(n)$ for a large positive integer $n$. 
However, if $f$ is a multiplicative function, then we can evaluate  $f(n)$
provided that the factorization of $n$ into prime powers and the formulas for
the multiplicative function $f$ at prime powers are known. 

Euler's totient function defined by
\begin{equation*}
  \varphi(n) =\ssum{k=1}{n}{(k,n)=1} 1,
\end{equation*}
which  is the count of positive integers that are less than $n$ and coprime
with $n$, is an important example of a multiplicative function.
For a proof of the multiplicative property of $\varphi(n)$ using the Chinese
Remainder Theorem, see \cite[Section 5.5]{Hardy-Wright}.
It is known that
\begin{equation*}
   \varphi(p^\alpha) = p^\alpha-p^{\alpha-1}  \quad \text{for}\ \alpha\geq1
\end{equation*}
and consequently, if
\begin{equation*}
   n=\prod_{j=1}^J p_j^{\alpha_j}
\end{equation*}
is the known factorization of $n$ into powers of distinct primes, then
\begin{equation*}
   \varphi(n)=\prod_{j=1}^J (p_j^{\alpha_j}-p_j^{\alpha_j-1}),  
\end{equation*}
and the values of $\varphi(n)$ can be computed in this way.

If $f$ and $g$ are two arithmetic functions, we define their Dirichlet
product $f*g$ to be the arithmetic function given by
\begin{equation*}
(f*g)(n)= \sum_{d|n} f(d)g(n/d),  
\end{equation*}
where $\displaystyle\sum_{d|n}f(d)$ denotes the sum of $f(d)$
over all positive divisors of $n$, and analogously for the union over divisors
denoted by \raisebox{0pt}[0pt][0pt]{$\displaystyle\bigcup_{d|n}$} below. 
The following result is well known:
\begin{theorem}\label{Dprod} 
If $f$ and $g$ are multiplicative, then $f*g$ is multiplicative.
\end{theorem}
\noindent For a proof, see \cite[Theorem 2.14]{Apostol}.
As an application of Theorem \ref{Dprod}, we have the following corollary,
which appears as an exercise in \cite[p. 49]{Apostol}:
\begin{corollary}\label{gcdsum}
Let $f$ be a multiplicative function. Then the function
\begin{equation*}
\xi_f(n)=\sum_{k=1}^n f\bigl((k,n)\bigr)  
\end{equation*}
is multiplicative.
\end{corollary}
\begin{proof}
Note that
\begin{equation*}
\{k|1\leq k\leq n\} =\bigcup_{d|n}\{k| 1\leq k\leq n\ \text{and}\ (k,n)=d\} 
\end{equation*}
and, therefore,
\begin{eqnarray*}
  \xi_f(n)&=&\sum_{k=1}^n f\bigl((k,n)\bigl)\;
   =\sum_{d|n}
  \ssum{k=1}{n}{(k,n)=d} f\bigl((k,n)\bigr)
  \\&=&\sum_{d|n}f(d)\varphi\big(n/d\big)=(f*\varphi)(n).
\end{eqnarray*}
Since $f$ and $\varphi$ are both multiplicative, we deduce by Theorem
\ref{Dprod} that $\xi_f=f*\varphi$ is multiplicative, too. 
\end{proof}

\section{A curious multiplicative function}
In \cite[Appendix C]{DEBZ}, T. Durt, B.-G. Englert, I. Bengtsson, and
K. \.{Z}yczkowski observed, during their study of the properties of mutually 
unbiased bases, the following interesting identity associated with the Gauss
sum: 
\begin{theorem}\label{Gauss-sum}
For any two integers $m$ and $n$ with ${0<m<n}$, and $\zeta_n=e^{2\pi i/n}$,
one has
\begin{equation*}
\frac{1}{\sqrt{n}}\left|\sum_{l=0}^{n-1}\zeta_{2n}^{(n-l)l m}\right|
=\begin{cases}
0\quad\text{if $\,n\,$ is even with both $\,n/(m,n)\,$ and $\,m/(m,n)\,$ odd,}
  \\[1ex]
  \sqrt{(m,n)}\quad\text{otherwise.}
 \end{cases}  
\end{equation*}
\end{theorem}
\noindent They arrived at this result by linear-algebra arguments and without
relying on the properties of Jacobi symbols and contour integrals that are
usually employed when evaluating Gauss sums.
This connection between linear algebra and number theory is noteworthy and,
therefore, we revisit the matter in Section~\ref{sec:MUB}.

In \cite{DEBZ}, they also indicated that
\begin{equation}\label{DEBZ_h}
h(n) =\frac{1}{\sqrt{n}}\sum_{m=1}^{n-1} \left|\sum_{l=0}^{n-1}
  \zeta_{2n}^{(n-l)l m}\right|
  +\left\{\begin{array}{@{}ll@{}}\displaystyle
            \frac{1}{2}\sqrt{n} & \text{if $n$ is even}\\[2ex]
            \sqrt{n} & \text{if $n$ is odd}
  \end{array}\right\}
\end{equation}
is a multiplicative function.
It is possible to simplify the expression of $h(n)$.
We first introduce $v_p(n)$, the count of prime number factors $p$ contained
in the positive integer $n$, defined by
\begin{equation*}
v_p(n) = \alpha \,\,\text{whenever $p^\alpha|n$ but $p^{\alpha+1}\nmid n$.}  
\end{equation*}
Then we can, after applying Theorem \ref{Gauss-sum}, rewrite $h(n)$ as
\begin{equation}\label{DEBZ-h}
  h(n)=\begin{cases} \displaystyle
    \sum_{k=1}^n \sqrt{(k,n)}-\frac{1}{2}\sqrt{n}
   -\ssum{k=1}{n-1}{v_2(k)=v_2(n)}\sqrt{(k,n)}
    &\text{if $n$ is even,}\\
   \displaystyle\sum_{k=1}^n \sqrt{(k,n)}&\text{if $n$ is odd.}
       \end{cases}  
\end{equation}
Note that by Corollary \ref{gcdsum},
\raisebox{0pt}[0pt][0pt]{$\displaystyle\xi_s(n)=\sum_{k=1}^n \sqrt{(k,n)}$}
is multiplicative because $s(n)=\sqrt{n}$ has this property.

The function $h(n)$ appears to be a new multiplicative function.
One question we can ask is whether we can construct multiplicative functions
when the prime $2$ that is privileged in the definition of $h(n)$ is replaced
by any odd prime $p$ and whether the function $s(n)$ can be replaced by other
arithmetic functions.
It turns out that this is possible.

\section{First generalization}
As we shall confirm in Section \ref{sec:confirm}, one generalization of $h(n)$
is the following: 
\begin{theorem} \label{DEBZ-general-2}
Choose a multiplicative function $f$ and a prime number $p$, as well as a
sequence of complex numbers $\kappa_0=0$, $\kappa_1$, $\kappa_2$, \dots,
and a sequence of positive integers $a_1$, $a_2$, \dots\  with
$a_{\alpha}\leq\alpha$.
Then the function
\begin{equation*}
  h^{(1)}_{f,p}(n)=\xi_f(n)-\kappa_{v_p(n)}\ssum{k=1}{n}{v_p(k)=v_p(n)-a_{v_p(n)}}
  f\bigl((k,n)\bigr)  
\end{equation*}
is multiplicative.
\end{theorem}
\begin{proof}
If $(m,n)=1$ and $(mn,p)=1$, then by Corollary \ref{gcdsum},
\begin{equation*}
h^{(1)}_{f,p}(mn)=\xi_f(mn)=\xi_f(m)\xi_f(n)=h^{(1)}_{f,p}(m)h^{(1)}_{f,p}(n).  
\end{equation*}
In order to complete the proof, we need to show that
if $\alpha>0$ and $(\nu,p)=1$, then
\begin{equation*}
  h^{(1)}_{f,p}(p^\alpha\nu) = h^{(1)}_{f,p}(p^\alpha)h^{(1)}_{f,p}(\nu).
\end{equation*}
We evaluate the sum in
\begin{equation*}
  h^{(1)}_{f,p}(p^\alpha \nu) =\xi_f(p^\alpha \nu)
  -\kappa_\alpha \ssum{k=1}{p^\alpha\nu}{v_p(k)=\alpha-a_\alpha}
  f\bigl((k,p^\alpha\nu)\bigr)
\end{equation*}
in a few steps, proceeding from
\begin{align*}
  \ssum{k=1}{p^\alpha\nu}{v_p(k)=\alpha-a_\alpha}
  f\bigl((k,p^\alpha\nu)\bigr)
&=f(p^{\alpha-a_{\alpha}})
  \ssum{k=1}{p^{a_{\alpha}}\nu}{(k,p)=1}f\bigl((k,p^{a_{\alpha}}\nu)\bigr)
=f(p^{\alpha-a_{\alpha}})
  \ssum{k=1}{p^{a_{\alpha}}\nu}{(k,p)=1}f\bigl((k,\nu)\bigr)\\
  &=f(p^{\alpha-a_{\alpha}})\sum_{k=1}^{p^{a_{\alpha}}\nu}
    f\bigl((k,\nu)\bigr)
    \biggl\lfloor\frac{1}{(k,p)}\biggr\rfloor,
\end{align*}
where $\lfloor x\rfloor$ denotes the floor of $x$, the largest integer that
does not exceed $x$.
Now, it is known that
\begin{equation*}
  \biggl\lfloor\frac{1}{n}\biggr\rfloor=\sum_{d|n}\mu(d)
  =\begin{cases}
   1 & \text{if $n=1$}\\ 0 & \text{if $n>1$} 
  \end{cases}
\end{equation*}
with the M\"obius function
\begin{equation*}
  \mu(n) = \begin{cases} 1 \quad\text{if $n=1,$} \\
(-1)^j\quad\text{if $n=p_1\cdots p_j,$ where the $p_i$s are distinct primes,}\\
0 \quad\text{otherwise,}\end{cases}
\end{equation*}
which we exploit in
\begin{equation*}
   \sum_{k=1}^{p^{a_{\alpha}}\nu}
    f\bigl((k,\nu)\bigr)
    \biggl\lfloor\frac{1}{(k,p)}\biggr\rfloor
=\sum_{k=1}^{p^{a_{\alpha}}\nu}
  f\bigl((k,\nu)\bigr)\sum_{\substack{d|k \\ d|p}} \mu(d)
  =\sum_{d|p}\mu(d)
    \sum_{\substack{k=1\\d|k}}^{p^{a_{\alpha}}\nu}f\bigl((k,\nu)\bigr).
\end{equation*}
Only $d=1$ and $d=p$ contribute to the sum, so that
\begin{align*}
   \sum_{k=1}^{p^{a_{\alpha}}\nu}
    f\bigl((k,\nu)\bigr)
    \biggl\lfloor\frac{1}{(k,p)}\biggr\rfloor
  &=\sum_{k=1}^{p^{a_{\alpha}}\nu}f\bigl((k,\nu)\bigr)
    -\sum_{k=1}^{p^{a_{\alpha}-1}\nu}f\bigl((k,\nu)\bigr)\\
  &=\bigl(p^{a_{\alpha}}-p^{a_{\alpha}-1}\bigr)\xi_f(\nu)
    =\varphi(p^{a_{\alpha}})\xi_f(\nu),
\end{align*}
where it is crucial that $a_{\alpha}\geq1$, and we arrive at
\begin{align*}
  h^{(1)}_{f,p}(p^\alpha \nu)
  &=\xi_f(p^\alpha \nu)
    -\kappa_\alpha f(p^{\alpha-a_{\alpha}})\varphi(p^{a_{\alpha}})\xi_f(\nu)
  \\&=\bigl[\xi_f(p^{\alpha})
-\kappa_\alpha
  f(p^{\alpha-a_{\alpha}})\varphi(p^{a_{\alpha}})\bigr]\xi_f(\nu)
  = h^{(1)}_{f,p}(p^{\alpha}) h^{(1)}_{f,p}(\nu)\,,
\end{align*}
where the last step recognizes that $h^{(1)}_{f,p}(p^{\alpha})=\xi_f(p^{\alpha})
-\kappa_\alpha f(p^{\alpha-a_{\alpha}})\varphi(p^{a_{\alpha}})$ as a
consequence of $\xi_f(1)=1$.
This completes the proof of Theorem~\ref{DEBZ-general-2}.
\end{proof}
\noindent%
Note that ${h^{(1)}_{f,p}(p^\alpha) %
  =\xi_f(p^{\alpha})-\kappa_\alpha\varphi(p^{\alpha})}$
when ${a_{\alpha}=\alpha}$.

\section{Second generalization.}
Here is another generalization of $h(n)$, also confirmed in Section
\ref{sec:confirm}: 
\begin{theorem} \label{DEBZ-general-3}
Choose a multiplicative function $f$ and a prime number $p$, as well as a
sequence of complex numbers $\kappa_0=0$, $\kappa_1$, $\kappa_2$, \dots,
and a sequence of nonnegative integers $a_1$, $a_2$, \dots\  with
$a_{\alpha}\leq\alpha$.
Then the function
\begin{equation*}
  h^{(2)}_{f,p}(n)
  =\xi_f(n)-\kappa_{v_p(n)}\ssum{k=1}{n}{v_p(k)\geq v_p(n)-a_{v_p(n)}}
  f\bigl((k,n)\bigr)  
\end{equation*}
is multiplicative.
\end{theorem}\enlargethispage{1.0\baselineskip}
\begin{proof}
Since $h^{(2)}_{f,p}(n)=\xi_f(n)$ when $(n,p)=1$, we need to show that
\begin{equation}\label{X}
  h^{(2)}_{f,p}(p^\alpha \nu) 
  = h^{(2)}_{f,p}(p^\alpha)\xi_f(\nu)
\end{equation}
for $\alpha>0$ and $(\nu,p)=1$.
We have
\begin{align*}
 &\hphantom{=\ } \ssum{k=1}{p^{\alpha}\nu}{v_p(k)\geq \alpha-a_{\alpha}}
  f\bigl((k,p^{\alpha}\nu)\bigr)
 =\sum_{k=1}^{p^{a_{\alpha}}\nu}f\bigl((p^{\alpha-a_{\alpha}}k,p^{\alpha}\nu)\bigr)
\\
  &={\left[\enskip\ssum{k=1}{p^{a_{\alpha}}\nu}{v_p(k)=0}\quad
    +\ssum{k=1}{p^{a_{\alpha}}\nu}{v_p(k)=1}+\cdots
    +\ssum{k=1}{p^{a_{\alpha}}\nu}{v_p(k)=a_{\alpha}-1}\qquad
    +\ssum{k=1}{p^{a_{\alpha}}\nu}{v_p(k)\geq a_{\alpha}}\quad\enskip
    \right]}f\bigl((p^{\alpha-a_{\alpha}}k,p^{\alpha}\nu)\bigr)\\
  &=\sum_{b=1}^{a_{\alpha}}
    \ssum{k=1}{p^{b}\nu}{v_p(k)=0}f\bigl((p^{\alpha-b}k,p^{\alpha}\nu)\bigr)
    +\sum_{k=1}^{\nu}f\bigl((p^{\alpha}k,p^{\alpha}\nu)\bigr)\\
  &=\sum_{b=1}^{a_{\alpha}}f\bigl(p^{\alpha-b}\bigr)
    \ssum{k=1}{p^{b}\nu}{(k,p)=1}f\bigl((k,\nu)\bigr)
    +f(p^{\alpha})\sum_{k=1}^{\nu}f\bigl((k,\nu)\bigr)\\
  &=\sum_{b=1}^{a_{\alpha}}f\bigl(p^{\alpha-b}\bigr)\varphi(p^b)\xi_f(\nu)
    +f(p^{\alpha})\xi_f(\nu),
\end{align*}
where we recall, from the proof of Theorem \ref{DEBZ-general-2}, that
\begin{equation*}
  \ssum{k=1}{p^b\nu}{(k,p)=1}f\bigl((k,\nu)\bigr)=\varphi(p^b)\xi_f(\nu)
  \quad\text{for $b>0$}.
\end{equation*}
It follows that \eqref{X} holds with
$h^{(2)}_{f,p}(p^\alpha)=\xi_f(p^{\alpha})-\kappa_{\alpha}
\raisebox{0pt}[0pt][0pt]{$\displaystyle\sum_{b=0}^{a_{\alpha}}$}
f\bigl(p^{\alpha-b}\bigr)\varphi(p^b)$,
which concludes the proof.
\end{proof}
\noindent%
Note that ${h^{(2)}_{f,p}(p^\alpha)=(1-\kappa_{\alpha})\xi_f(p^{\alpha})}$
when ${a_{\alpha}=\alpha}$.

\section{Generalizations confirmed}\label{sec:confirm}
We now confirm that $h^{(1)}_{f,p}(n)$ and $h^{(2)}_{f,p}(n)$, introduced in
Theorems \ref{DEBZ-general-2} and \ref{DEBZ-general-3}, are generalizations of
$h(n)$ in \eqref{DEBZ-h}.

In view of the important role played by the privileged prime $p$, we write
$n=p^{v_p(n)}\bigl(p^{-v_p(n)}n\bigr)$  and note that
\begin{align*}
  h(n)&=h\bigl(2^{v_2(n)}\bigr)\xi_s\bigl(2^{-v_2(n)}n\bigr),\\
  h^{(1)}_{f,p}(n)&=h^{(1)}_{f,p}\bigl(p^{v_p(n)}\bigr)
                    \xi_f\bigl(p^{-v_p(n)}n\bigr),\\
  h^{(2)}_{f,p}(n)&=h^{(2)}_{f,p}\bigl(p^{v_p(n)}\bigr)
                    \xi_f\bigl(p^{-v_p(n)}n\bigr),
\end{align*}
with
\begin{eqnarray}\label{R1}\notag
  h\bigl(2^{\alpha}\bigr)&=&\xi_s\bigl(2^{\alpha}\bigr)-2^{\alpha/2-1}
                           \quad\text{for $\alpha\geq1$},\\
  h^{(1)}_{f,p}\bigl(p^{\alpha}\bigr)&=&\xi_f\bigl(p^{\alpha}\bigr)
     -\kappa_\alpha f\bigl(p^{\alpha-a_{\alpha}}\bigr)
                                       \varphi\bigl(p^{a_{\alpha}}\bigr),
  \notag\\
  h^{(2)}_{f,p}(p^\alpha)&=&\xi_f(p^{\alpha})-\kappa_{\alpha}
  \sum_{b=0}^{a_{\alpha}}f\bigl(p^{\alpha-b}\bigr)\varphi(p^b).
\end{eqnarray}
Therefore, we obtain ${h_{f,p}^{(1)}(n)=h(n)}$ and ${h_{f,p}^{(2)}(n)=h(n)}$
for $f(n)=\sqrt{n}=s(n)$ and $p=2$, if we choose $\kappa_{\alpha}$ in
accordance with
\begin{align*}
  \kappa_{\alpha}&=2^{-a_{\alpha}/2}\quad\text{for $a_{\alpha}\geq1$ in}\
                   h^{(1)}_{s,2}(n),\\
  \kappa_{\alpha}&=\Bigl(2^{1+a_{\alpha}/2}+2^{(1+a_{\alpha})/2}-2^{1/2}\Bigr)^{-1}
                   \quad\text{for $a_{\alpha}\geq0$ in}\ h^{(2)}_{s,2}(n),
\end{align*}
for all $\alpha\geq0$.
Since these assignments work for all permissible choices for the
$a_{\alpha}$s, we have multiple generalizations of $h(n)$ from both
$h_{f,p}^{(1)}(n)$ and $h_{f,p}^{(2)}(n)$, although the conditions
${v_p(k)=v_p(n)-a_{v_p(n)}}$ in Theorem \ref{DEBZ-general-2} and 
${v_p(k)\geq v_p(n)-a_{v_p(n)}}$ in Theorem \ref{DEBZ-general-3} do not really
match the condition $v_2(k)=v_2(n)$ in \eqref{DEBZ-h} because $a_{\alpha}=0$
is not allowed in Theorem \ref{DEBZ-general-2}.
Clearly, then, \eqref{DEBZ-h} is just one of many ways of rewriting $h(n)$ of
\eqref{DEBZ_h}.

\section{Each \lowercase{$h^{(1)}_{f,p}$} is a \lowercase{$h^{(2)}_{f,p}$} and %
  vice versa} 
The mappings ${(f,p)\mapsto h^{(1)}_{f,p}}$ and ${(f,p)\mapsto h^{(2)}_{f,p}}$
are both characterized by a sequence of $\kappa_{\alpha}$s and a
sequence of $a_{\alpha}$s, and---for given $f$ and~$p$---one can
always adjust the sequences of one of them to the sequences of the other such
that the right-hand sides in \eqref{R1} are the same,
${h^{(1)}_{f,p}\bigl(p^{\alpha}\bigr)=h^{(2)}_{f,p}\bigl(p^{\alpha}\bigr)}$.
In this sense, all the multiplicative functions in the $h^{(1)}_{f,p}$
family are also contained in the $h^{(2)}_{f,p}$ family, and vice versa,
although the two mappings are really different. 

To justify this remark, we shall write $\kappa^{(1)}_{\alpha}$ and
$a^{(1)}_{\alpha}$ for the parameters that specify $h^{(1)}_{f,p}$ and
$\kappa^{(2)}_{\alpha}$ and $a^{(2)}_{\alpha}$ for those of $h^{(2)}_{f,p}$.
Then, for a particular choice of the $\kappa^{(1)}_{\alpha}$s and
$a^{(1)}_{\alpha}$s, we put $\kappa^{(2)}_{\alpha}=0$ if
$\kappa^{(1)}_\alpha f\bigl(p^{\alpha-a^{(1)}_{\alpha}}\bigr)=0$;
otherwise we choose either $a^{(2)}_{\alpha}=a^{(1)}_{\alpha}$ or
$a^{(2)}_{\alpha}=a^{(1)}_{\alpha}-1\geq0$ such that
$F=\raisebox{0pt}[\height][10pt]%
   {$\displaystyle\sum_{b=0}^{a^{(2)}_{\alpha}}$}%
   f\bigl(p^{\alpha-b}\bigr)\varphi(p^b)\neq0$
and put
$\kappa^{(2)}_{\alpha}=F^{-1}\kappa^{(1)}_\alpha %
f\bigl(p^{\alpha-a^{(1)}_{\alpha}}\bigr)\varphi\bigl(p^{a^{(1)}_{\alpha}}\bigr)$.
Conversely, for a particular choice of the $\kappa^{(2)}_{\alpha}$s and
$a^{(2)}_{\alpha}$s, we put $a^{(1)}_{\alpha}=\alpha$ and
$\kappa^{(1)}_{\alpha}=\varphi\bigl(p^{\alpha}\bigr)^{-1}\kappa^{(2)}_{\alpha}%
  \raisebox{0pt}[\height][0pt]{$\displaystyle\sum_{b=0}^{a^{(2)}_{\alpha}}$}%
  f\bigl(p^{\alpha-b}\bigr)\varphi(p^b)$.
These assignments ensure that
${h^{(1)}_{f,p}\bigl(p^{\alpha}\bigr)=h^{(2)}_{f,p}\bigl(p^{\alpha}\bigr)}$.

\section{A linear-algebra proof of Theorem \ref{Gauss-sum}}\label{sec:MUB}
We revisit here the linear-algebra proof of Theorem \ref{Gauss-sum}.
While we follow the reasoning in \cite{DEBZ}, where Theorem~\ref{Gauss-sum} is
a side issue and the ingredients are widely scattered, the presentation here
is self-contained and adopts somewhat simpler conventions, in particular for
the phase factor in the definition of $\MC_m$ in \eqref{def-Cm} below.

\heading{Getting started: Columns and rows, matrices, eigenvector bases}
We consider column vectors with $n$ complex entries (${n\geq2}$), their
adjoint row vectors, and the $n\times n$ matrices that implement linear
mappings of columns to columns and rows to rows.
For any two columns $\col{x}$ and $\col{y}$, we denote the adjoint rows by
$\col{x}\adj$ and $\col{y}\adj$;
a row-times-column product such as $\col{x}\adj\col{y}$ is a complex number
that can be understood as the inner product of the columns $\col{x}$ and
$\col{y}$, or of the rows $\col{x}\adj$ and $\col{y}\adj$, whereby
$\col{x}\adj\col{x}\geq0$ with ``$=$'' only for $\col{x}=\col{0}$. 
The row-times-column products such as $\col{y}\col{x}\adj$ are $n\times n$
matrices with $\tr{\col{y}\col{x}\adj}=\col{x}\adj\col{y}$ for the matrix trace.
We recall that $(\col{x}\adj)\adj=\col{x}$,
$(\col{x}\adj\col{y})\adj=\col{y}\adj\col{x}$, 
and $(\col{y}\col{x}\adj)\adj=\col{x}\col{y}\adj$.

Following H.~Weyl \cite[Sec.~IV.D.14]{Weyl} and J.~Schwinger
\cite[Sec.~1.14]{Schwinger}, our basic ingredients are two related
unitary $n\times n$ matrices $\MA$ and $\MB$ of period $n$, that is
\begin{equation*}
  \MA^k=\unity,\quad \MB^k=\unity\quad
  \text{if $k\equiv0\ (\text{mod}\ n)$ and only then},
\end{equation*}
where $\unity$ is the $n\times n$ unit matrix.
The eigenvalues of $\MA$ and $\MB$ are the powers of $\zeta_n$, the basic $n$th
root of unity that appears in Theorem \ref{Gauss-sum}.
We denote the $j$th eigencolumn of $\MA$ by $\ca_j$ and the $k$th eigencolumn of
$\MB$ by $\cb_k$, 
\begin{equation*}
  \MA\ca_j=\ca_j\,\zeta_n^j,\quad \MB\cb_k=\cb_k\,\zeta_n^k,\qquad
  \ca_j\adj\MA=\zeta_n^j\,\ca_j\adj,\quad\cb_k\adj\MB=\zeta_n^k\,\cb_k\adj,
\end{equation*}
where we regard the labels as modulo-$n$ integers, so that $\ca_{j+n}=\ca_j$
and $\cb_{k+n}=\cb_k$. 
The sets of eigencolumns are orthonormal and complete,
\begin{align*}
  \text{orthonormality:}&\quad
  \ca_j\adj \ca_k\phadj=\delta^{(n)}_{j,k}
  =\left\{
  \begin{array}{@{}l@{\enskip}l@{}}
    1 & \text{if}\ \zeta_n^j=\zeta_n^k \\[1ex]
    0 & \text{if}\ \zeta_n^j\neq\zeta_n^k
  \end{array}\right\}
  =\frac{1}{n}\sum_{l=0}^{n-1}\zeta_n^{(j-k)l}=\cb_j\adj \cb_k\phadj,\\
    \text{completeness:}&\quad
           \sum_{j=0}^{n-1}\ca_j\phadj \ca_j\adj=\unity
          =\sum_{k=0}^{n-1}\cb_k\phadj \cb_k\adj,
\end{align*}
where $\delta^{(n)}_{j,k}$ is the modulo-$n$ version of the Kronecker delta
symbol, and the projection matrices associated with the eigenvectors are
\begin{equation}\label{A1}
  \ca_j\phadj \ca_j\adj=\frac{1}{n}\sum_{k=0}^{n-1}
                         \Bigl(\zeta_n^{-j}\MA\Bigr)^k,\quad
  \cb_k\phadj \cb_k\adj=\frac{1}{n}\sum_{l=0}^{n-1}\Bigl(\zeta_n^{-k}\MB\Bigr)^l.
\end{equation}
The unitary matrices $\MA$ and $\MB$ are related to each other by the discrete
Fourier transform that turns one set of eigenvectors into the other,
\begin{equation*}
  \ca_j\adj \cb_k\phadj=\frac{1}{\sqrt{n}}\zeta_n^{jk},
  \quad \cb_k\phadj=\frac{1}{\sqrt{n}}\sum_{j=0}^{n-1}\ca_j\zeta_n^{jk},
    \quad \ca_j\adj=\frac{1}{\sqrt{n}}\sum_{k=0}^{n-1}\zeta_n^{jk}\cb_k\adj.
\end{equation*}
As a consequence, we have the following identities:
\begin{eqnarray}\label{A2}
 & \ca_j\adj \MB= \ca_{j+1}\adj,\quad \MA\cb_k=\cb_{k+1},\quad
   \tr{\MA^j\MB^k}=n\,\delta^{(n)}_{j,0}\,\delta^{(n)}_{k,0},&\nonumber\\
    &\zeta_n^{jk}\MA^j\MB^k=\MB^k\MA^j,\quad
    (\MA^j\MB^k)^l=\zeta_{2n}^{jk(l-1)l}\MA^{jl}\MB^{kl};&
\end{eqnarray}
we leave their verification to the reader as an exercise.

\heading{More unitary matrices and their eigenstate bases}
For $m=1,2,\dots,n-1$, we define
\begin{equation*}\label{def-Cm}
  \MC_m=\zeta_{2n}^{-(n-1)m}\MA\MB^m ,
\end{equation*}
which are unitary matrices of period $n$, 
\begin{equation}\label{A3}
  \MC_m^k=\zeta_{2n}^{-(n-1)mk}(\MA\MB^m)^k=\zeta_{2n}^{-(n-k)mk}\MA^k\MB^{mk}
  \xrightarrow{k=n}\unity.
\end{equation}
Upon denoting the $j$th eigencolumn of $\MC_m$ by $\cc_{m,j}=\cc_{m,j+n}$, we
have $\MC_m\cc_{m,j}=\cc_{m,j}\zeta_n^j$ and
$\cc_{m,j}\adj\MC_m=\zeta_n^j\cc_{m,j}\adj$ as well as 
\begin{equation}\label{A4}
  \cc_{m,j}\adj \cc_{m,k}\phadj=\delta^{(n)}_{j,k},\quad
  \sum_{j=0}^{n-1} \cc_{m,j}\phadj \cc_{m,j}\adj=\unity,\quad
  \cc_{m,j}\phadj \cc_{m,j}\adj
  =\frac{1}{n}\sum_{k=0}^{n-1}\Bigl(\zeta_n^{-j}\MC_m\Bigr)^k. 
\end{equation}

To establish how the $\cc_{m_j}$s are related to the $\ca_j$s and $\cb_k$s, we
first infer $\cc_{m,j}\adj \cb_k\phadj$ from the recurrence relation
\begin{equation*}
  \cc_{m,j}\adj \cb_{k+1}\phadj= \cc_{m,j}\adj \MA \cb_k\phadj
  = \cc_{m,j}\adj\zeta_{2n}^{(n-1)m}\MC_m\MB^{-m}\cb_k\phadj
  = \cc_{m,j}\adj \cb_k\phadj \,\zeta_{2n}^{(n-1)m}\,\zeta_n^{j-mk},
\end{equation*}
which yields
\begin{equation*}
   \cc_{m,j}\adj \cb_k\phadj
  =\cc_{m,j}\adj \cb_0\phadj \,\zeta_{2n}^{(n-1)mk}\,\zeta_n^{jk-mk(k-1)/2}
  = \frac{1}{\sqrt{n}}\,\zeta_n^{jk}\zeta_{2n}^{(n-k)km},
\end{equation*}
where we adopt $\cc_{m,j}\adj \cb_0\phadj=1/\sqrt{n}$ by convention.
Then we exploit the completeness of the $\cb_k$s in
\begin{equation}\label{A-Gauss}
  \ca_j\adj \cc_{m,k}\phadj
  =\sum_{l=0}^{n-1}\ca_j\adj \cb_l\phadj \cb_l\adj \cc_{m,k}\phadj
   =\frac{1}{n}\sum_{l=0}^{n-1}\zeta_n^{(j-k)l}\zeta_{2n}^{-(n-l)lm}.
\end{equation}
This Gauss sum is our first ingredient.

Next we utilize ${\col{x}\adj\col{y}=\tr{\col{y}\col{x}\adj}}$ and the
linearity of the trace as well as statements in \eqref{A1}--\eqref{A4} in 
\begin{align}\label{A6}\notag
  \bigl| \ca_j\adj \cc_{m,k}\phadj\bigr|^2
  &= \ca_j\adj \cc_{m,k}\phadj \cc_{m,k}\adj \ca_j\phadj
    = \tr{ \ca_j\phadj \ca_j\adj\; \cc_{m,k}\phadj \cc_{m,k}\adj}\\ \notag
  &= \frac{1}{n^2}\sum_{l,l'=0}^{n-1}
    \tr{ \Bigl(\zeta_n^{-j}\MA\Bigr)^{l'}\Bigl(\zeta_n^{-k}\MC_m\Bigr)^l}\\ \notag
  &= \frac{1}{n^2}\sum_{l,l'=0}^{n-1}
    \zeta_n^{-jl'-kl}\zeta_{2n}^{-(n-l)lm}
    \tr{\MA^{l'+l}\MB^{ml}}\\
  &= \frac{1}{n}\sum_{l=0}^{n-1}
    \zeta_n^{(j-k)l}\zeta_{2n}^{-(n-l)lm}\delta^{(n)}_{ml,0}.
\end{align}
Now writing $d=(m,n)$, $n=d\nu$, $m=d\mu$ with $(\mu,\nu)=1$, we have
\begin{equation*}
  \delta^{(n)}_{ml,0}=\delta^{(\nu)}_{l,0}
\end{equation*}
because ${ml\equiv0\ (\text{mod}\ n)}$ requires
${l\equiv0\ (\text{mod}\ \nu)}$.
Therefore, only the terms with $l=0,\nu,2\nu,\ldots,(d-1)\nu$ contribute
to the final sum in \eqref{A6}, and we arrive at
\begin{align*}
 n \bigl| \ca_j\adj \cc_{m,k}\phadj\bigr|^2
 & =\sum_{l=0}^{n-1}\zeta_n^{(j-k)l}\zeta_{2n}^{-(n-l)lm}\delta^{(\nu)}_{l,0}
   =\sum_{l=0}^{d-1}\zeta_n^{(j-k)l\nu}\zeta_{2n}^{-(n-l\nu)l\nu m}\\
  & =\sum_{l=0}^{d-1}\zeta_d^{(j-k)l}(-1)^{(d-1)\mu\nu l}
  = \left\{\begin{array}{@{}l@{\quad}l@{}}
     d\,\delta^{(d)}_{j,k+d/2} & \text{if $(d-1)\mu\nu$ is odd,}\\[1ex]
     d\,\delta^{(d)}_{j,k} & \text{if $(d-1)\mu\nu$ is even.}
    \end{array}\right.
\end{align*}
Upon combining this second ingredient with the first in \eqref{A-Gauss}, we
conclude that
\begin{equation*}
  \frac{1}{\sqrt{n}}\left|\sum_{l=0}^{n-1}
    \zeta_n^{(j-k)l}\zeta_{2n}^{(n-l)lm}\right|
  =\left\{\begin{array}{@{}l@{\enskip}l@{}}
            \sqrt{(m,n)}\,\delta^{((m,n))}_{j,k+(m,n)/2}
            & \text{if $v_2(n)=v_2(m)\geq1$}\\[1ex]
            \sqrt{(m,n)}\,\delta^{((m,n))}_{j,k}  & \text{otherwise.}
    \end{array}\right.\quad
\end{equation*}
For $j=k$, this is the statement in Theorem \ref{Gauss-sum}.

\clearpage  

In passing, we found the following identity between the absolute value of a
Gauss sum and a particular partial sum:
\begin{corollary}\label{Gauss-sum-identity}
For all integers $k,m,n$ with $n\geq1$, we have
  \begin{equation*}
    \left|\frac{1}{n}\sum_{l=0}^{n-1}
      \zeta_n^{kl}\zeta_{2n}^{(n-l)lm}
      \right|^2
    =\frac{1}{n}\sum_{l=0}^{n-1}
    \zeta_n^{kl}\zeta_{2n}^{(n-l)lm}\delta^{(n)}_{ml,0}.
  \end{equation*}
\end{corollary}
\begin{proof}
For $m$ values in the range ${0<m<n}$, this follows from comparing
\eqref{A-Gauss} and \eqref{A6}, and the case of ${m=0}$ is immediate.
Then, the observation that $\zeta_n^{kl}\zeta_{2n}^{(n-l)lm}$ does not change
if we replace $m$ by $m\pm n$, in conjunction with replacing $k$ by
$k+\frac{1}{2}n$ if $n$ is even, extends the permissible $m$ values to all
integers.  

The identity can, of course, also be verified directly.
We note that $\zeta_n^{kl}\zeta_{2n}^{(n-l)lm}$ does not change when $l$ is
replaced by $l\pm n$ and, therefore, the summation over $l$ can cover any range
of $n$ successive integers. 
Accordingly, we can replace $l$ by $l+a$ for any integer $a$ without changing
the value of the sum. 
We exploit this when writing the left-hand side as a double sum
and then processing it,
\begin{align*}
  &\;\frac{1}{n}\sum_{l'=0}^{n-1}
                      \zeta_n^{-kl'}\zeta_{2n}^{-(n-l')l'm}
              \frac{1}{n}\sum_{l=0}^{n-1}
                      \zeta_n^{kl}\zeta_{2n}^{(n-l)lm}
  \\=&\;\frac{1}{n^2}\sum_{l,l'=0}^{n-1}
              \zeta_n^{-kl'}\zeta_{2n}^{-(n-l')l'm}
              \zeta_n^{k(l+l')}\zeta_{2n}^{(n-l-l')(l+l')m}
  \\=&\;\frac{1}{n}\sum_{l=0}^{n-1}\zeta_n^{kl}\zeta_{2n}^{(n-l)lm}
       \frac{1}{n}\sum_{l'=0}^{n-1}\zeta_n^{-ll'm}
     =\frac{1}{n}\sum_{l=0}^{n-1}
          \zeta_n^{kl}\zeta_{2n}^{(n-l)lm}\delta^{(n)}_{ml,0},
\end{align*}
thereby arriving at the right-hand side.
\end{proof}

\heading{Remark: Unbiased bases}
The eigenvector bases for the matrices $\MA$ and $\MB$ are such that
$\bigl|\ca_j\adj \cb_k\phadj\bigr|$ has the same value for all rows
$\ca_j\adj$ and columns $\cb_k$, which is the defining property of a pair of
\emph{unbiased} bases. 
Further, for each $m$, the basis
$\mathcal{C}_m=\mathop{\{\cc_{m,k}\}}_{k=1}^{n}$ is unbiased with the basis
$\mathcal{B}=\mathop{\{\cb_k\}}_{k=1}^n$.
When $n$ is prime, each $\mathcal{C}_m$ is also unbiased with the basis
$\mathcal{A}=\mathop{\{\ca_j\}}_{j=1}^n$. 
When $n$ is not prime, however, then some of the $\mathcal{C}_m$s
are unbiased with $\mathcal{A}$, namely those with $(m,n)=1$, and the others
are not.
Further, two bases $\mathcal{C}_m$ and $\mathcal{C}_{m'}$ with $m'>m$ are
unbiased if $(m'-m,n)=1$, and only then, because
$\cc_{m,j}\adj\cc_{m',k}\phadj =\ca_j\adj\cc_{m'-m,k}\phadj$.
We refer the reader to \cite{DEBZ} for a detailed discussion of unbiased bases.

\enlargethispage{1.0\baselineskip} 

\section{Yet another multiplicative function}
Here we report one more multiplicative function that is suggested by
$h(n)$ in \eqref{DEBZ-h}, but is not a generalization in the spirit of
$h^{(1)}_{f,p}(n)$ and $h^{(2)}_{f,p}(n)$ in Theorems \ref{DEBZ-general-2} and
\ref{DEBZ-general-3}.

In preparation, and for the record, we note that the Gauss sum in
Theorem~\ref{Gauss-sum},
\begin{equation}\label{Smn}
  S(m,n)=\frac{1}{\sqrt{n}} \sum_{l=0}^{n-1} \zeta_{2n}^{(n-l)l m},
\end{equation}
can be evaluated.
We write $\nu=n/(m,n)$ and $\mu=m/(m,n)$ as above and distinguish three cases:
\begin{enumerate}
\item[(a)]  If $v_2(n)=0$ or $v_2(m)>v_2(n)>0$, then $\nu$ is odd and $(n-1)\mu$
is even, and we have
\begin{equation*}
  S(m,n)
  =\begin{cases}
  \ \bigl(\frac{1}{2}(n-1)\mu|\nu\bigr)\sqrt{(m,n)}
  & \text{if $i^{\nu}=i$,}\\[1ex]
  i\bigl(\frac{1}{2}(n-1)\mu|\nu\bigr)\sqrt{(m,n)}
  & \text{if $i^{\nu}=-i$.}
  \end{cases}
\end{equation*}
\item[(b)] If $v_2(n)>v_2(m)$, then $(n-1)\mu$ is odd, and we have
\begin{equation*}
  S(m,n)=e^{\pm i\pi/4}\bigl(2\nu|(n-1)\mu\bigr)\sqrt{(m,n)}  
  \quad\text{for}\  i^{(n-1)\mu}=\pm i.
\end{equation*}
\item[(c)] If $v_2(n)=v_2(m)\geq 1$, we have
\begin{equation*}
   S(m,n) =0.
\end{equation*}
\end{enumerate}
Here $(j|k)$ is the familiar Jacobi symbol
\cite[Sec.~9.7]{Apostol}; all $(j|k)$ appearing here are
equal to $+1$ or $-1$ because the arguments are co-prime in each case.
We leave the proof to the reader as an exercise in the evaluation of Gauss
sums, and the sum in Corollary~\ref{Gauss-sum-identity} for $k\neq0$ is
another exercise.
The facts  required to complete these proofs can be found in
\cite[Sec.~1.5]{Berndt-Evans-Williams}.
The linear-algebra argument in Section~\ref{sec:MUB} establishes the absolute
value of the right-hand side, namely $\sqrt{(m,n)}$ in (a) and (b) and $0$ in
(c), but not its complex phase.

Now, harking back to $h(n)$ in \eqref{DEBZ-h}, we observe that
\begin{equation*}
  h(n) = \sum_{m=1}^{n} \sqrt{(m, n)} = \sum_{m=1}^{n} |S(m, n)|
  \quad\text{for $n$ odd,}
\end{equation*}
and there is an extra term when $n$ is even.
Since $\displaystyle\sum_{m=1}^{n} \sqrt{(m,n)}$ is a multiplicative function
of $n$ for all positive $n$, odd or even, it is natural for us to ask if
\begin{equation*}
    s(n)=\sum_{m=1}^{n} S(m,n)
\end{equation*}
is a multiplicative function of $n$,
too.
While this turns out to be false (see for example the case of $n=6=2\times3$),
it \emph{is} true for odd $n$, for which $s(n)$ is real.
For even $n$, an extra term is required in addition to discarding the
imaginary part of  $s(n)$, reminiscent of the even-$n$ modification in
\eqref{DEBZ-h}:
\begin{theorem}\label{h-sharp}
With the Gauss sum $S(m,n)$ in \eqref{Smn}, the function
  \begin{equation*}
    h^{\#}(n)= \sum_{m=1}^{n} \mathrm{Re}\bigl(S(m,n)\bigr)
      +\left\{\begin{array}{@{}ll@{}}\displaystyle
            \frac{1}{2}\sqrt{n} & \text{if $n$ is even}\\[2ex]
            0 & \text{if $n$ is odd}
  \end{array}\right\}
  \end{equation*}
is multiplicative.
\end{theorem}
\begin{proof}
We evaluate the sum over $m$ in $s(n)$ by expressing $S(m,n)$ in terms of the
sum over $l$ in \eqref{Smn} and carrying out the $m$ summation first,
\begin{align*}
  \sqrt{n}\,s(n)
 &=\sum_{m=1}^{n}\sum_{l=0}^{n-1} \zeta_{2n}^{(n-l)l m}\\
 &=\sum_{l=0}^{n-1}\sum_{m=1}^{n}
   \Bigl[\delta^{(2n)}_{(n-l)l,0}+\Bigl(1-\delta^{(2n)}_{(n-l)l,0}\Bigr)\Bigr]
   \zeta_{2n}^{(n-l)l m}\\
 &=\sum_{l=0}^{n-1}\Biggl[\delta^{(2n)}_{(n-l)l,0}\,n
   -\Bigl(1-\delta^{(2n)}_{(n-l)l,0}\Bigr)\frac{1-\zeta_{2n}^{(n-l)ln}}
                                             {1-\zeta_{2n}^{-(n-l)l}}\Biggr],
\end{align*}
where we separate the terms with $\zeta_{2n}^{(n-l)l}=1$ from the others.
Next, we note that
\begin{equation*}
  \frac{1}{2}\Bigl(1-\delta^{(2n)}_{(n-l)l,0}\Bigr)
              \Bigl(1-\zeta_{2n}^{(n-l)ln}\Bigr)
    =\frac{1}{2}\Bigl(1-\delta^{(2n)}_{(n-l)l,0}\Bigr)
    \Bigl(1-(-1)^{(n-1)l}\Bigr)
    =\delta^{(2)}_{n,0}\,\delta^{(2)}_{l,1}
\end{equation*}
and
\begin{equation*}
  \mathrm{Re}\left(\frac{2}{1-\zeta_{2n}^{-(n-l)l}}\right)=1.
\end{equation*}
Therefore, 
\begin{equation*}
    \sqrt{n}\,\mathrm{Re}\bigl(s(n)\bigr)
  = n\sum_{l=0}^{n-1}\delta^{(2n)}_{(n-l)l,0}-\dfrac{n}{2} \delta^{(2)}_{n,0},  
\end{equation*}
and we deduce that
\begin{equation*}
  h^{\#}(n)=\mathrm{Re}\bigl(s(n)\bigr)
             +\frac{1}{2}\sqrt{n}\,\delta^{(2)}_{n,0}
             =\sqrt{n}\sum_{l=0}^{n-1}\delta^{(2n)}_{(n-l)l,0}
             =\sqrt{n}\sum_{l=1}^{n}\delta^{(2n)}_{(n-l)l,0}.
\end{equation*}
We now proceed to show that $h^{\#}(n)$ is multiplicative.

We first consider the case of $n$ odd.
Since  
\begin{equation*}
  (n-l)l\equiv (1-l)l\equiv 0 \pmod{2}
\end{equation*}
holds when $n$ is odd, and
\begin{equation*}
  (n-l)l\equiv -l^2\pmod{n}
\end{equation*}
for all $n$, we have
\begin{equation}\label{alt-d}
  \delta_{(n-l)l,0}^{(2n)}=\delta_{(n-l)l,0}^{(n)}=\delta_{l^2,0}^{(n)}.
\end{equation}
Now write $n=\lambda\nu^2$, where $\lambda$ is squarefree, 
or equivalently that $|\mu(\lambda)|=1$.
The fact that $n|l^2$ implies that $\nu^2|l^2$ or $\nu|l$.
Therefore, if we write $l=\nu s$ then $n|l^2$ implies that $\lambda|s^2$.
But if $p|\lambda$ then $p|s^2$ and this implies that $p|s$.
Since $\lambda$ is squarefree, we conclude that $\lambda|s$ and consequently,
$l=\nu \lambda \omega$ for some positive integer $\omega$.
Since $l\leq n$, we conclude that $\omega\leq \nu$.
This implies that
\begin{equation}\label{alt-id}
  \sum_{l=1}^{n}\delta^{(n)}_{l^2,0} = \nu.
\end{equation}
We remark that \eqref{alt-id} is true for any positive integer $n$.
Combining \eqref{alt-id} and \eqref{alt-d}, we deduce that
\begin{equation*}
  h^{\#}(n) = \sqrt{n}\sum_{l=1}^{n}\delta^{(2n)}_{(n-l)l,0}
            = \sqrt{n}\sum_{l=1}^{n}\delta^{(n)}_{l^2,0}
            = \sqrt{n}\,\nu.  
\end{equation*}

Turning to $n$ even now, we write $n=2m$ with $m>0$ and observe that
\begin{align*}
   \frac{1}{\sqrt{2m}}h^{\#}(2m)
  &=\sum_{l=1}^{2m}\delta^{(4m)}_{(2m-l)l,0}
       =\ssum{l=1}{2m}{l\ \text{even}}\delta^{(4m)}_{(2m-l)l,0}\notag\\
  &=\ssum{k=1}{m}{l=2k}\delta^{(4m)}_{(2m-l)l,0}
      =\sum_{k=1}^{m}\delta^{(m)}_{(m-k)k,0}
  =\sum_{k=1}^{m}\delta^{(m)}_{k^2,0}.
\end{align*}
By \eqref{alt-id}, we conclude that
\begin{equation*}
\sum_{k=1}^m \delta^{(m)}_{k^2,0}=\eta,  
\end{equation*}
where $\eta$ is given by $m=\dfrac{n}{2}=\kappa \eta^2$, with squarefree
integer $\kappa$.
This implies that
\begin{equation*}
   h^{\#}(n) = h^{\#}(2m) = \sqrt{n}\,\eta. 
\end{equation*}

Now, let $a$ and $b$ be two positive integers with $(a,b)=1$.
If $a$ and $b$ are both odd, we may write $a=\lambda_a^{\ }\nu_a^2$ and
$b=\lambda_b^{\ }\nu_b^2$ with squarefree integers $\lambda_a$ and $\lambda_b$. 
Then
\begin{equation*}
h^{\#}(ab) = h^{\#}\bigl(\lambda_a\lambda_b(\nu_a\nu_b)^2\bigr)
=\sqrt{ab}\,\nu_a\nu_b = h^{\#}(a)h^{\#}(b).  
\end{equation*}
Next suppose either $a$ or $b$ is even.
We may assume that $a$ is even and $b$ is odd.
Write $a/2= \lambda_a^{\ }\nu_a^2$ and $b=\lambda_b^{\ }\nu_b^2$ with squarefree
integers $\lambda_a$ and $\lambda_b$; then
\begin{equation*}
h^{\#}(ab) = \sqrt{ab}\,\nu_a\nu_b. 
\end{equation*}
Now $h^{\#}(a) = \sqrt{a}\,\nu_a$ and $h^{\#}(b) =\sqrt{b}\,\nu_b$ and hence
\begin{equation*}
h^{\#}(ab) =h^{\#}(a)h^{\#}(b),
\end{equation*}
and this completes the proof that $h^{\#}(n)$ is multiplicative.
\end{proof}

\heading{Three remarks} 
(i)  Although the multiplicative function $h^{\#}$ is constructed
differently, and is \emph{another} multiplicative function in this sense, 
it is also contained in the families $h^{(1)}_{f,p}$ and $h^{(2)}_{f,p}$ of
Theorems \ref{DEBZ-general-2} and \ref{DEBZ-general-3}, because every
multiplicative function is in these families for a corresponding $f$.
This follows from the fact that the mapping $f\to\xi_f=f*\varphi$ is
invertible, and we can choose $\kappa_{\alpha}=0$ for all $\alpha$.
In particular, we have $h^{\#}=\xi_{f^{\#}}$ with the multiplicative function
$f^{\#}(n)$ specified by its prime-power values, that is 
\begin{equation*}
  h^{\#}\bigl(2^{1+\alpha}\bigr)=2^{(1+\alpha)/2+\lfloor\alpha/2\rfloor}\,,
  \quad f^{\#}\bigl(2^{1+\alpha}\bigr)=\frac{2^{1/2}-1}{3}
  \bigl[1-(-2)^{1+\alpha}\bigr]
\end{equation*}
for powers of $2$, and
\begin{equation*}
  h^{\#}\bigl(p^{\alpha}\bigr)=p^{\alpha/2+\lfloor\alpha/2\rfloor}\,,
  \quad
  f^{\#}\bigl(p^{\alpha}\bigr)
  =1+\frac{p-p^{1/2}}{p+1}\bigl[(-p)^{\alpha}-1\bigr]
\end{equation*}
for powers of odd primes, where $\alpha\geq0$.\newline
(ii)  It is striking that the $m$-sums of both the absolute value and the real
part of $S(m,n)$ yield multiplicative functions after a suitable modification
for even $n$.
This makes us wonder if there are other functions of the pair $m,n$ with this
property.\newline
(iii) Although, right now, we do not know truly useful applications of any
particular multiplicative functions in the families $h^{(1)}_{f,p}$ and
$h^{(2)}_{f,p}$, it is worth recalling that $h(n)$ of \eqref{DEBZ_h}
is closely related to a prime-distinguishing function \cite{DEBZ}.
Similarly, the value of $h^{\#}(n)$ tells us the squarefree factor $\lambda$
in $n=\lambda\nu^2$,
\begin{equation*}
  \lambda=\left\{
    \begin{array}{cl}
      \displaystyle\biggl(\frac{n}{2h^{\#}(n)}\biggr)^2
                   & \mbox{if both $n$ and $v_2(n)$ are even,}\\[2.5ex]
      \displaystyle\biggl(\frac{n}{h^{\#}(n)}\biggr)^2
                   & \mbox{otherwise,}
    \end{array}\right.
\end{equation*}
which is a corollary to the proof of Theorem~\ref{h-sharp}.
However, it is only the current lack of an efficient algorithm for the
evaluation of the Gauss sum in \eqref{Smn} that prevents $h(n)$ and $h^{\#}(n)$
from being practical tools. 

\section{Summary}
Inspired by the peculiar multiplicative function $h(n)$ in \eqref{DEBZ-h}, we
found two mappings that turn a given multiplicative function into other
multiplicative functions, with each image function specified by a privileged
prime number, a sequence of complex numbers, and a sequence of nonnegative
integers.
In addition, we reported one more multiplicative function, of a different
kind, also suggested by the structure of $h(n)$.

\bigskip\noindent\textbf{Acknowledgments.}
We sincerely thank Si Min Chan for her role in getting our collaboration
started.
We are grateful to Ron Evans for indicating the existence of explicit formulas
for $S(m,n)$ and his encouragement. 
B.-G.~E. is grateful for the encouragement by Thomas Durt, Ingemar Bengtsson,
and Karol \.Zyczkowski, the co-authors of the 2010 review article~\cite{DEBZ}. 
The Centre for Quantum Technologies is a
Research Centre of Excellence funded by the Ministry of
Education and the National Research Foundation of Singapore.

\bigskip\noindent\textbf{Dedication.}
It is a pleasure to dedicate this article to Professor Bruce Berndt on the
occasion of his 80th birthday.

\end{document}